\documentclass[12pt]{article}

\usepackage{amsmath}
\usepackage{amssymb}
\usepackage{amsfonts}
\usepackage{amsthm}

\newtheorem{theorem}{Theorem}
\newtheorem{definition}[theorem]{Definition}

\newtheorem{example}{Example}

\usepackage[body={8.5in,11in},margin=1in]{geometry}
\usepackage{enumerate}
\usepackage{setspace}
\usepackage{graphicx}
\usepackage[all]{xy}
\newenvironment{itemize*}%
  {\begin{itemize}%
    \setlength{\itemsep}{0pt}%
    \setlength{\parskip}{0pt}}%
  {\end{itemize}}
\newenvironment{enumerate*}%
  {\begin{enumerate}%
    \setlength{\itemsep}{0pt}%
    \setlength{\parskip}{0pt}}%
  {\end{enumerate}}

\DeclareMathSymbol{\Z }{\mathbin}{AMSb}{"5A}
\DeclareMathSymbol{\C}{\mathbin}{AMSb}{"43}
\DeclareMathSymbol{\N}{\mathbin}{AMSb}{"4E}
\DeclareMathSymbol{\Q}{\mathbin}{AMSb}{"51}
\DeclareMathSymbol{\R}{\mathbin}{AMSb}{"52}

\title{Irreducibles in the Integers modulo $n$}
\author{James Lanterman \\ \small{ Supervised by Dr. Jeremiah Reinkoester}}
\date{}
\begin{document}
\maketitle 

\begin{section}{Introduction}
\quad In 2011, Anderson and Frazier introduced a general theory of factorization of elements in integral domains \cite{anderson}.
Given a relation $\tau$ on an integral domain $D$, they defined a $\tau$-factorization of an element $a\in D$ by $a=\lambda a_1...a_n$ where $\lambda$ is a unit in $D$ and $a_i \tau a_j$ for all $i, j$. They briefly investigated the irreducible and prime elements of the integers under the congruence modulo $n$ relation (denoted $\tau_n$). This paper further investigates the irreducible integers under the $\tau_n$ relation for particular values of $n$ in the hopes of finding a general form in which to express them. We were successful in finding all irreducible integers under the $\tau_n$ relation for $n= 2, 3, 4, 5, 6, 7,$ and $11$ by making use of another equivalence relation based on $\tau_n$, and we were able to find a general form for all these irreducibles in Theorem \ref{generalization}.
\end{section}
\begin{section}{Buildup}
We begin with a few definitions for clarity. 
\begin{definition}{Two integers x and y are said to be \emph{$\tau_n$-related}, denoted x $\tau_n$ y, if, and only if, x $\equiv$ y (mod n).} \end{definition}
\begin{definition}{For an integer $x, x=\lambda a_1 a_2...a_k$ is called a \emph{$\tau_n$-factorization of $x$} if $\lambda\in U(\Z)$ and $a_i\tau_n a_j$ for all $i, j$. We say the $\tau_n$-factorization is \emph{proper} if $k>1$.} \end{definition}
\begin{definition}{If a proper $\tau_n$-factorization of an integer $x$ does not exist, $x$ is a \emph{$\tau_n$-atom}.}\end{definition}
\begin{definition}{A positive integer $x$ is a \emph{$\tau_n$-prime} if, whenever $x$ divides a $\tau_n$-factorization $\lambda a_1...a_k$, then $x$ divides $a_i$ for some $i$.} \end{definition}
To avoid confusion, we shall use the term ``usual prime" when referencing the standard idea of prime numbers in the integers. It is worth mentioning that, clearly, all of the usual primes are $\tau_n$-primes.

An example is in order at this point to ensure understanding. 

\begin{example}\emph{Consider the integer $98=2*7*7$. Below are  3 possible factorizations of $98$: 
\[
	\begin{array}{c}
	2*7*7\\
	2*49\\
	-14*-7
	\end{array}
\]
}
\emph{Since $-7\tau_7 -14$, then $-14*-7$ is a $\tau_7$-factorization of $98$. However, none of these factorizations are $\tau_2$-factorizations, and indeed none exist; thus $98$ is a $\tau_2$-atom. This warrants the question: is $98$ a $\tau_2$-prime? Notice that $196=14*14$, which is clearly a $\tau_2$-factorization, and $98|196$, but $98\not| 14$: thus, $98$ is not a $\tau_2$-prime. $14$, however, is a $\tau_2$-prime: notice that if $14$ divides some $\tau_2$-factorization $p_1p_2...p_k$, then $2$ must also divide it, and so $p_i\tau_2 0$ for some $i$. Since $p_i\tau_2 p_j$ for all $i,j$, then each term in this product must be $\tau_2$-related to $0$; that is, they are all divisible by $2$. Further, since $14$ divides the product, then $7$ must also divide it, and since $7$ is a usual prime it must divide some $p_m$. Since both $2$ and $7$ divide this $p_m$, then $14|p_m$, and so $14$ is a $\tau_2$-prime. }
\end{example}

This example illustrates that some $\tau_n$-atoms may not be $\tau_n$-primes. However, $\tau_n$-primes are all, in fact, $\tau_n$-atoms; this is shown in \cite{anderson}. An alternative proof is given below.

\begin{theorem}\label{primeatom} If an integer $x$ is a $\tau_n$-prime, then $x$ is a $\tau_n$-atom.
\end{theorem}

\begin{proof}
Let integer $x$ be a $\tau_n$-prime. Then if $x$ divides some $\tau_n$-factorization $\pm p_1p_2...p_k$, $x|p_i$ for some $i$. Suppose, by way of contradiction, that $x$ is not a $\tau_n$-atom. Then there must exist some proper $\tau_n$-factorization for $x$; denote it by $\pm x_1x_2...x_m$. Notice, then, that $x|\pm x_1x_2...x_m$. Hence, $x|x_j$ for some $j$. However, since $x_1x_2...x_m$ is a \emph{proper} $\tau_n$-factorization, then necessarily $x_j<x$ for all $j$, and so $x\nmid x_j$ for all $j$. A contradiction arises; thus, $x$ must be a $\tau_n$-atom. 
\end{proof}

In their paper \cite{anderson}, Anderson and Frazier explored the $\tau_2$-atoms and $\tau_2$-primes in particular. They were able to show that the $\tau_2$-primes were of the form $2p$ for usual prime $p\neq 2$, and that the $\tau_2$-atoms were of the form $2p_1p_2...p_k$ for usual primes $p_i\neq 2$.

Anderson and Frazier were also able to find a general form for the $\tau_n$-primes for any value of $n$:
\begin{theorem} \label{primechar} An integer $b$ is a $\tau_n$-prime if, and only if, $b=p_1^{e_1}p_2^{e_2}...p_k^{e_k}q$, where $e_i=1$ or $0$ for all $i$, $p_j$ is a usual prime which divides $n$ for all $j$, and $q$ is a usual prime which does not divide $n$ or $q=1$.
\end{theorem}

While this result is quite good, the $\tau_n$-primes are but a subset of the $\tau_n$-atoms. Thus, we investigated the $\tau_n$-atoms for different values of $n$ in the interest of finding a general form, similar to that for $\tau_n$-primes in Theorem \ref{primechar}.
\end{section}

\begin{section}{The $\tau_3$-atoms, $\tau_4$-atoms, and $\tau_6$-atoms}

\quad It is logical to start with the next least complicated structure, that of the integers under the $\tau_3$-relation. We already know that the usual primes are $\tau_n$-atoms for all values of $n$ by Theorem \ref{primeatom}, but (as was the case with the $\tau_2$-atoms) there are almost certainly more. For example, we know by Theorem \ref{primechar} that $3p$ is a $\tau_3$-prime (and thus a $\tau_3$-atom) for all usual primes $p$. However, when investigating the $\tau_3$-atoms, an interesting complication arose: that of factorizations involving negative factors. Since $U(\Z)=\{\pm 1\}$, then any number of factors can be negated merely be introducing the appropriate number of factors of $-1$. For example, consider the proper factorizations of $28$. It may be tempting to assume that only $3$ exist (namely, $2*2*7, 4*7,$ and $2*14$), but in fact $13$ possible proper factorizations exist by simply negating different combinations of factors. This complication, however, is easily solved by recalling that $-k\tau_n(n-k)$ for all $k, n$. Thus, if an integer $x$ is $\tau_3$-related to $2$, then $-x\tau_3 1$. With this in mind, we can show the following:

\begin{theorem}\label{tau3atoms} The $\tau_3$-atoms consist of the usual primes and integers of the form $3p_1p_2...p_m$, where $p_i$ is a prime not equal to $3$ for all $i$. \end{theorem}

\begin{proof}
By Theorem \ref{primeatom}, we know that the usual primes are $\tau_3$-atoms. 

Let $k$ be a $\tau_3$-atom with usual prime factorization $p_1p_2...p_m$, $m>1$. We shall proceed by considering cases based on the multiplicity of $3$ in $k$.

\emph{Case 1}: Suppose $3$ does not divide $k$. Then for all $i$, $p_i\tau_3 1$ or $p_i\tau_3 2$. Suppose that $b$ of the $p_i$ terms are $\tau_3$-related to $1$ and $m-b$ of the $p_i$ terms are $\tau_3$-related to $2$. Then if $m-b$ is even, simply negating all the $p_i$ terms $\tau_3$-related to $2$ will yield a $\tau_3$-factorization; similarly, if $m-b$ is odd, then negate all the $p_i$ terms $\tau_3$-related to $2$ and introduce a factor of the unit $-1$, yielding a $\tau_3$-factorization. Either way, $k$ has been shown to have a $\tau_3$-factorization, yet $k$ is a $\tau_3$-atom by hypothesis - a contradiction. Thus, $3$ must divide $k$. 

\emph{Case 2}: Suppose $3$ divides $k$ twice or more. Then the usual prime factorization of $k$ is of the form $3*3p_3...p_m$. But $3*(3p_3...p_m)$ is a $\tau_3$-factorization of $k$, since both factors are $\tau_3$-related to $0$; thus, $k$ is not a $\tau_3$-atom, giving rise to another contradiction. Thus, $3$ must divide $k$ exactly once. It only remains to be shown that any integer for which this holds is indeed a $\tau_3$-atom.

\emph{Case 3}: Suppose $3$ divides $k$ exactly once. Then the usual prime factorization of $k$ is of the form $3p_2...p_m$, where $p_i\neq 3$ for all $i$. Notice that, in any grouping of these factors, whichever factor is divisible by $3$ must be $\tau_3$-related to $0$, while all the other factors cannot possibly be $\tau_3$-related to $0$, since they are necessarily not divisible by $3$. Thus, no $\tau_3$-factorization for $k$ exists, and so $k$ is a $\tau_3$-atom. 
\end{proof}

Notice that the $\tau_3$-atoms are quite similar to the $\tau_2$-atoms: we know the $\tau_2$-atoms (other than the usual primes) are of the form $2p_1p_2...p_k$ where $p_i$ is a prime not equal to $2$, while the above theorem shows that the $\tau_3$-atoms are of the form $3q_1q_2...q_k$ where $q_i$ is a prime not equal to $3$. At this point, one cannot help but wonder whether an integer of the form $np_1p_2...p_k$ where $n$ does not divide $p_1p_2...p_k$ is a $\tau_n$-atom for any $n$. This is not always the case (consider that $8$ divides $16$ exactly once, as desired, yet $4*4$ is a $\tau_8$-factorization of $16$), but it \emph{does} always hold when $n$ is a usual prime. 

\begin{theorem}\label{onefactoratom} For a positive prime integer $n$, any integer of the form $np_1p_2...p_k$ where $n$ does not divide $p_1p_2...p_k$ (that is, an integer in which the multiplicity of $n$ is exactly 1) is a $\tau_n$-atom. If the multiplicity of $n$ is greater than $1$ in any integer, then that integer must not be a $\tau_n$-atom.
\end{theorem}

\begin{proof}
Let $x$ be an integer with prime factorization $np_1p_2...p_k$ where $n$ does not divide $p_1p_2...p_k$. Since $n$ is a usual prime, then, similar to Case 3 of the proof of Theorem \ref{tau3atoms}, it can be seen that in any grouping of the factors of $x$, one factor (namely, the one which $n$ divides) must be $\tau_n$-related to $0$, while the others necessarily must not be $\tau_n$-related to $0$. Thus, no $\tau_n$-factorization of $x$ exists, and so $x$ must be a $\tau_n$-atom.

Let $y$ be an integer in which the multiplicity of $n$ is at least $2$. Then, as in Case 2 of Theorem \ref{tau3atoms}, it can be seen that we may simply take the $\tau_n$-factorization $n*(y/n)$, both of which must be $\tau_n$-related to $0$, and so $y$ is not a $\tau_n$-atom. Notice that $n$ need not be prime for this to be true.
\end{proof}

Because of how nice these properties are, we will allow $n$ to be a usual prime for the remainder of the paper unless stated otherwise. Much insight can be gained into the $\tau_n$-atoms for composite values of $n$ based on this, however, given the following result:

\begin{theorem}\label{divisibleatom}If an integer $x$ is a $\tau_n$-atom, and $n|m$ for some integer $m\geq n$, then $x$ is a $\tau_m$-atom. 
\end{theorem}

\begin{proof}
Let $n, m$ be positive integers such that $n|m$ and let $x$ be a $\tau_n$-atom. Then in any factorization $x=\pm a_1a_2...a_k$, there must be some factors $a_i$ and $a_j$ such that $a_i\not\tau_n a_j$; that is, $n\nmid a_i-a_j$. Then since $n|m$, $m\nmid a_i-a_j$, and so $a_i \not \tau_m a_j$. Thus there must be no $\tau_m$-factorization of $x$, and so $x$ is a $\tau_m$-atom.
\end{proof}

This actually allows us to find all the $\tau_4$- and $\tau_6$-atoms. 

\begin{theorem}\label{tau4atoms} An integer $x$ is a $\tau_4$-atom if, and only if, $x$ is a $\tau_2$-atom. 
\end{theorem}

\begin{proof}
We know by Theorem \ref{divisibleatom} that a $\tau_2$-atom is a $\tau_4$-atom. Suppose, by way of contradiction, that an integer $x$ is a $\tau_4$-atom and not a $\tau_2$-atom. Recall that the $\tau_2$-atoms are the usual primes and integers in which the multiplicity of $2$ is exactly one. Then there are two cases.

\emph{Case 1:} $2$ divides $x$ more than once; denote this by $x=2^{j+1}p_1p_2...p_k$, where $p_i$ is an odd prime for all $i$. Notice that $2$ multiplied by any odd number returns an even value not divisible by $4$; that is, an integer $\tau_4$-related to $2$. Thus, we may simply write $x$ as $2*2*...*(2p_1p_2...p_k)$, where there are $j$ factors of $2$ before the final factor. Then all these factors are $\tau_4$-related to $2$, and thus we have a $\tau_4$-factorization for $x$, which produces a contradiction. 

\emph{Case 2:} $2$ does not divide $x$. Then all prime factors of $x$ must be odd, and so must all be $\tau_4$-related to either $1$ or $3$. Notice that $-1\tau_4 3$; thus, as in Case 1 of Theorem \ref{tau3atoms}, we can simply negate all the factors of $x$ that are $\tau_4$-related to $1$ and, if necessary, include a factor of $-1$ to produce a $\tau_4$-factorization of $x$, producing a contradiction.

Therefore any $\tau_4$-atom must also be a $\tau_2$-atom.
\end{proof}

\begin{theorem}\label{tau6atoms} An integer $x$ is a $\tau_6$-atom if, and only if, $x$ is either a $\tau_2$-atom or a $\tau_3$-atom.
\end{theorem}
\begin{proof}
Again, by Theorem \ref{divisibleatom} we know that the $\tau_2$- and $\tau_3$-atoms must be $\tau_6$-atoms. Suppose, by way of contradiction, that an integer $x$ is a $\tau_6$-atom that is neither a $\tau_2$-atom nor a $\tau_3$-atom; that is, the multiplicity of neither $2$ nor $3$ in $x$ is exactly 1, and $x$ is not a usual prime. The following 4 cases exhaust all possibilities, then:

\emph{Case 1:} The multiplicities of both $2$ and $3$ in $x$ is zero; that is, $x=p_1p_2...p_k$ where $p_i$ is a usual prime not equal to $2$ or $3$ for all $i$. Consider some arbitrary usual prime $p_j$ which divides $x$. Clearly $p_j\not\tau_6 0$, else $6|p_j$ and so $p_j$ is not a usual prime. Nor can $p_j\tau_6 2$ be true, else $p_j=6q+2$ for some integer $q$, and so $2|p_j$ and, since $p_j\neq 2$, $p_j$ is not a usual prime. Similarly, $p_j\not\tau_6 3$ and $p_j\not\tau_6 4$. Thus, $p_j\tau_6 1$ or $p_j\tau_6 5$, and since $p_j$ is arbitrary, this holds for all $j$. Notice, however, that $-1\tau_6 5$, and so, similar to Case 1 of Theorem \ref{tau3atoms}, we may simply negate all usual prime factors of $x$ and, if necessary, introduce a factor of $-1$ to produce a $\tau_6$-factorization of $x$. Thus, either $2$ or $3$ must have multiplicity at least $1$ in $x$, and since their multiplicity cannot be exactly one by hypothesis, it must be greater than one.

\emph{Case 2:} The multiplicity of $2$ in $x$ is greater than $1$, and the multiplicity of $3$ in $x$ is zero; that is, $x=2^yp_1p_2...p_k$ where $p_i$ is a usual prime not equal to $2$ or $3$ for all $i$ and $y>1$. The product of any number of the $p_i$ factors cannot be $\tau_6$-related to $3$, else that product would be equal to $6q+3$ for some integer $q$, and thus would necessarily be divisible by $3$; but $3\not |x$. Further, by the same logic as Case 1, none of the $p_i$ factors can be $\tau_6$-related to either $2, 4$, or $6$, else they must not be usual primes. Thus, they must all be $\tau_6$-related to either $1$ or $5$. Notice that the product of all the $p_i$ factors $\tau_6$-related to $1$ is still $\tau_6$-related to $1$; thus, take their product and call this product $d_1$. Now, since $5\tau_6 -1$, then the product of all the $d_i$ factors $\tau_6$-related to $5$ must be $\tau_6$-related to either $1$ or $-1$; take this product and call it $x_5$. Then $x=2^yx_1x_5$. If $x_5\tau_6 1$, then notice that $x=2*2*...*(2d_1d_5)$ is a $\tau_6$-factorization, since $d_1d_5\tau_6 1$, and so $2d_1d_5\tau_6 2$. If $d_5\tau_6 -1$, then, similarly, $x=2*2*...*(-2d_1d_5)*-1$ is a $\tau_6$-factorization, since $-1$ is a unit.

\emph{Case 3:} The multiplicity of $3$ in $x$ is greater than $1$, and the multiplicity of $2$ in $x$ is zero; that is, $x=3^yp_1p_2...p_k$ where $p_i$ is a usual prime not equal to $2$ or $3$ for all $i$ and $y>1$. The proof of this case is similar to that of Case 2.

\emph{Case 4:} The multiplicities of both $2$ and $3$ in $x$ are greater than $1$. Then $x=2^y3^zp_1p_2...p_k$, where $p_i$ is a usual prime not equal to $2$ or $3$ for all $i$ and $y, z>1$. Then notice that $x=6*(2^{y-1}3^{z-1}p_1p_2...p_k)$ is a $\tau_6$-factorization of $x$.

Thus, $x$ must be either a $\tau_2$-atom or $\tau_3$-atom.
\end{proof}

\end{section}
\begin{section}{The $\mu_n$ relation}

\quad For higher values of $n$, determining the $\tau_n$-atoms becomes significantly more difficult, in part due to the issue of negative factors. For example, $6$ may seem to be a $\tau_5$-atom, since $6=2*3$ and $2\not\tau_5 3$, but notice that $6=-1*-2*3$, and since $-1$ is a unit and $-2\tau_5 3$, then this is a $\tau_5$-factorization of $6$. This problem becomes significantly more difficult to overcome when considering integers with very many factors. We will solve this problem by introducing the $\mu_n$ relation.

Following the observation that $-k\tau_n (n-k)$, we write the following definition:
\begin{definition} For two integers $x$ and $y$, $x$ is \emph{$\mu_n$-related} to $y$, denoted $x\mu_n y$, if $x\tau_n\pm y$.
\end{definition}
In this way, we can simply worry about whether a factorization exists in which all terms are $\mu_n$-related, eliminating the need to consider the unit $-1$. 

Since the $\mu_n$ relation is based on the $\tau_n$ relation, which we know is an equivalence relation, it is of interest whether $\mu_n$ is an equivalence relation as well, and it should be rather clear to the reader that it is. As a result, we can consider the equivalence classes of the integers under the $\mu_n$ relation, which we will denote (at this time) using the familiar notation $[x]=\{y\in\Z|y\mu_n x\}$. Notice that, as there are $n$ equivalence classes of the $\tau_n$ relation, there must be $\lceil n/2\rceil$ equivalence classes of the $\mu_n$ relation, and since we are only concerned with when $n$ is a usual prime, $\lceil n/2\rceil=\dfrac{n+1}{2}$. 

\begin{theorem}\label{mureps} For a usual prime $n>2$, and for any positive integer $a$ such that $1<a<n, \left\{[0], [a], [a^2], ... [a^{\frac{n-1}{2}}]\right\}$ is the set of all equivalence classes of the integers under the $\mu_n$ relation. Further, $a^{\frac{n-1}{2}}\mu_n 1$.
\end{theorem}

\begin{proof}
Since $n$ is a usual prime, it is relatively prime to any $a<n$. Thus, by Fermat's Little Theorem, we can say that, for such an $a<n$, $\left\{[0], [a], [2a], ... [\frac{n-1}{2} a]\right\}$ is a set of $\frac{n+1}{2}$ distinct equivalence classes of the integers under the $\tau_n$-relation. 
We claim that these are also distinct equivalence classes of the integers under the $\mu_n$-relation. Suppose not: that is, suppose $sa \mu_n ra$ for some nonnegative integers $r, s\leq\frac{n-1}{2}.$ Then $sa\tau_n\pm ra$, and so $n|(s\pm r)a$. Since $0<a<n$, then $n\nmid a$; thus $n|s\pm r$. Notice that $|s\pm r|\leq n-1$, and thus $n|s\pm r$ if, and only if, $s\pm r=0$. Then $s=\pm r$, and since $s$ and $r$ are nonnegative then $s=r$. Thus, the aforementioned equivalence classes are indeed distinct in the integers under the $\mu_n$ relation, and further, since there are $\frac{n+1}{2}$ such classes, these must represent \emph{all} of the equivalence classes of the integers under the $\mu_n$ relation.

Notice that the set $S=\left\{1, 2, ..., \frac{n-1}{2}\right\}$ has a representative from each equivalence class except $[0]$. Then the product of all the elements of $S$ must be $\mu_n$-related to the product $\left\{b, 2b, ..., \frac{n-1}{2}b\right\}$ for any positive integer $b< n$; that is, $\left(\frac{n-1}{2}\right)!\tau_n\pm\left(\frac{n-1}{2}\right)!*b^{\frac{n-1}{2}}$. Since $\frac{n-1}{2}!$ is relatively prime to $n$, then it must have an inverse modulo $n$. Thus, $\pm1\tau_n b^{\frac{n-1}{2}}$. Since $b$ was arbitrary, this holds for any positive integer less than $n$.

Consider the set $\left\{[0], [a], [a^2], ... [a^{\frac{n-1}{2}}]\right\}$ for some arbitrary positive integer $a<n$. It will be shown that this, too, is the set of all distinct equivalence classes of the integers under the $\mu_n$ relation. Clearly none of the classes represented by the powers of $a$ must be equivalent to that represented by $[0]$, since $n$ is a usual prime. Suppose that $a^s\mu_n a^r$ where $r<s\leq\frac{n-1}{2}$. Then $a^{s-r}\mu_n 1$, so $a^{s-r}\tau_n \pm 1$. Suppose that $(s-r)$ is the smallest positive integer such that $a^{s-r}\tau_n\pm 1$. By the Quotient Remainder Theorem, $\frac{n-1}{2}=(s-r)q+k$ for some positive integers $q$ and $k$, with $k<(s-r)$. Then since, as stated in the previous paragraph, $a^{\frac{n-1}{2}}\tau_n\pm 1$, then consider that $\pm 1\tau_n a^{\frac{n-1}{2}}\tau_n a^{(s-r)q+k}=a^{(s-r)q}a^k\tau_n (a^{(s-r)})^qa^k\tau_n (\pm 1)^q a^k\tau_n \pm a^k$; in short, $\pm 1\tau_n a^k$. But $k<(s-r)$, a contradiction. Thus, $sa\not\mu_n ra$, and so the equivalence classes must be distinct. 
\end{proof}

For any usual prime $n$, notice that if an integer $x$ has usual prime factorization $x=p_1p_2...p_k$, and $p_i\mu_n 0$ for any $i$, then we already know by Theorem \ref{onefactoratom} whether $x$ is a $\tau_n$-atom: if only one factor is in the $\mu_n$ equivalence class $[0]$ then it is a $\tau_n$-atom, and if more than one is, it is not. Thus, for the remainder of the paper we will only be concerned with integers whose usual prime factorizations contain no factors $\mu_n$-related to $0$. 

At this point, we will introduce a new notation to further simplify the interpretation of usual prime factorizations of integers in the context of $\tau_n$-factorizations by indexing the equivalence classes of the $\mu_n$ relation. First, let $y$ and $z$ be arbitrary integers with $z\mu_n 1$. Clearly $y\in [y]$ and $z\in[z]$, but consider $yz$. Notice that, since $z\mu_n 1$, then $yz\mu_n y$. In essence, when we multiply integers by elements of $[1]$, we do not change equivalence classes. We shall denote an arbitrary element of $[1]$ by $x_0$. This notation will become clearer in time.

Now we know by Theorem \ref{mureps} that for any integer $a$ such that $1<a<n$ where $n$ is a usual prime, $a^{\frac{n-1}{2}}\mu_n \pm 1$. We will denote elements of the equivalence class $[a]$ by $x_1$, of $[a^2]$ by $x_2$, and so on. This indexing is not unique to each usual prime $n$, but relies only on a choice of the integer $a$ for that value of $n$. Notice that, for a given usual prime $n$, these indexing values will range from $0$ to $\frac{n-3}{2}$.

We shall now demonstrate the helpful nature of this indexing system we have built up. Suppose that we are curious about whether some integer $x$ is a $\tau_n$-atom for some large usual prime $n$, and $x$ has a usual prime factorization $x=x_1x_2x_3$; that is, for some integer $a\in(1, n),$ the usual prime factorization of $x$ has one factor which is $\mu_n$-related to $a$, one which is $\mu_n$-related to $a^2$, and one $\mu_n$-related to $a^3$. Now, clearly, $x_1x_2x_3$ is not a $\tau_n$-factorization of $x$. However, consider the product $x_1x_2$. Since $x_1\mu_n a$ and $x_2\mu_n a^2$, then $x_1x_2\mu_n a^2a=a^3$. But $x_3\mu_n a^3$, so $x_1x_2\mu_n x_3$. Then $x_1x_2\tau_n\pm x_3$, and so $x$ has a $\tau_n$-factorization; namely, either $x=(x_1x_2)*x_3$ or $x=-1*(x_1x_2)*x_3$. Notice, though, that $x_1x_2\mu_n x_3$, and (looking at the indices) $1+2=3$. This is done to create an easier, addition-based indexing which allows us to combine terms in an easy and predictable manner. In general, in fact, $x_ix_j\mu_n x_{i+j \emph{ mod } \frac{n-1}{2}}$, which should be clear since $x_ix_j\mu_n a^ia^j=a^{i+j}\mu_n x_{i+j}$; the "mod" clause is simply inserted to ensure that we use the smallest representative, as $a^{\frac{n-1}{2}}\mu_n 1$. We shall approach multiplication of these representatives by considering their indices under addition mod $\frac{n-1}{2}$. We will adopt one last convention of notation for the remainder of the paper: the factorization of $x$ mentioned earlier, $x=(x_1x_2)*x_3$, will be written as $x=x_3*x_3$. This is not intended to imply that there is a single factor with a multiplicity of $2$ in $x$, but instead to merely state that there are $2$ factors in $x$ which are both in the $\mu_n$ equivalence class $[a^3]$ for some integer $a\in(1,n)$. Henceforth, it should not be assumed that all integers denoted $x_i$ are equal for a particular integer $i$, but instead that both are merely within the same $\mu_n$ equivalence class.

We shall use the results of this section to investigate the $\tau_n$-atoms for higher values of $n$. The next section is appropriately shorter than the previous sections, as this new notation streamlines the process of finding $\tau_n$-atoms greatly.
\end{section}

\begin{section}{The $\tau_5$-atoms}

\quad The next usual prime after $3$ is $5$, and so we will naturally move on at this point to investigate the $\tau_5$-atoms. By Theorem \ref{onefactoratom}, we know that integers of the form $5p_1p_2...p_k$, where $p_i$ is a usual prime not equal to $5$, are $\tau_5$-atoms, along with the usual primes. However, given the example at the beginning of the previous section, $6$ is also a $\tau_5$-atom, yet meets neither of these criteria, and so we aim to characterize the remaining $\tau_5$-atoms. By Theorem \ref{onefactoratom}, we can ignore any other integers which are divisible by $5$; thus, for the remainder of the section we only consider those integers not divisible by $5$ unless otherwise stated. Notice that there are only two $\mu_5$ equivalence classes other than $[0]$: $[1]$ and $[2]$; thus, we will denote elements of $[1]$ by $x_0$ and elements of $[2]$ by $x_1$.

Clearly any usual prime factorization of an integer with no $x_1$ factors must contain only $x_0$ elements, and so such an integer is not a $\tau_5$-atom. Similarly, we need not consider those integers whose usual prime factorizations have no $x_0$ factors. There must be some $x_0$ and some $x_1$ factors. 

\begin{theorem}\label{tau5atoms}The $\tau_5$-atoms are the usual primes, integers whose usual prime factorizations are of the form $5p_1p_2...p_k$ where $p_i$ is a usual prime not equal to $5$, and integers whose usual prime factorizations are of the form $x_0*x_0*...*x_0*x_1$, where each $x_0$ is a usual prime in the $\mu_n$ equivalence class $[1]$ and $x_1$ is a usual prime in the $\mu_n$ equivalence class $[2]$.
\end{theorem}

\begin{proof}
It suffices to show that integers with prime factorizations of the form $x_0*x_0*...*x_0*x_1$ are $\tau_5$ atoms, and that no other integers which are neither usual primes nor divisible by $5$ are $\tau_5$-atoms. Recall, again, that all factors denoted $x_0$ are \emph{not necessarily equal}; this notation merely communicates the $\mu_n$ equivalence class of each factor.

Let $x$ be an integer with usual prime factorization $x=x_0*x_0*...*x_0*x_1$. First, note that this is not a $\tau_5$-factorization. Further, since $0+0=0$, the product of any number of $x_0$ factors is simply another $x_0$ factor, and since $0+1=1$, then the product $x_0*x_1$ is another $x_1$ factor. Thus, there must always be exactly one $x_1$ factor in any factorization of $x$, and so $x$ must not have a $\tau_5$-factorization, as there must either be only $x_0$ factors or multiple $x_1$ factors in a $\tau_5$-factorization.. Thus, $x$ is a $\tau_5$-atom.

To show that there are no overlooked $\tau_5$-atoms, let $y$ be an integer with usual prime factorization $y=x_0*x_0*...*x_0*x_1*x_1*...x_1$; that is, any integer which is neither a usual prime nor divisible by $5$, and which has more than one $x_1$ factor. Then since $0+0=0$, we may simply multiply all of the $x_0$ factors together into a single $x_0$ factor; that is, $y=x_0*x_1*x_1*...*x_1$. Then, since $0+1=1$, multiplying the $x_0$ factor by a single $x_1$ factor produces another $x_1$ factor, and so $y=(x_0*x_1)*x_1*...*x_1=x_1*x_1*...*x_1$. Thus, $y$ has a $\tau_n$-factorization, since it has a factorization in which all factors share a $\mu_n$ equivalence class. 
\end{proof}

This exhausts all possibilities. A $\tau_5$-atom that is neither a usual prime nor divisible by $5$ must have at least one $x_1$ factor, but it cannot have more than one; thus it must have exactly one. It must have at least one $x_0$ factor (else it is a usual prime, a contradiction), however it may have any number of $x_0$ factors in addition to the obligatory $x_1$ factor. 
\end{section}

\begin{section}{The $\tau_7$-atoms}
\quad Next, we move on to the $\tau_7$-atoms. Again, we know by Theorem \ref{onefactoratom} that integers of the form $7p_1p_2...p_k$ where $p_i$ is a usual prime not equal to $7$ are $\tau_7$-atoms, along with the usual primes. Notice that there are $3$ nonzero equivalence classes of the $\mu_7$ relation; thus, we shall be calling their representatives $x_0, x_1,$ and $x_2$, similar to the way we denoted the $\mu_5$ equivalence classes. 

Recall that by Theorem \ref{tau5atoms}, integers whose usual prime factorizations are of the form \\$x_0*x_0*...*x_0*x_1$ are $\tau_5$-atoms (where $x_0$ and $x_1$ refer to $\mu_5$ equivalence classes). When we consider this kind of integer in reference to $\mu_7$ equivalence classes, we have two possibilities: $x_0*x_0*...x_0*x_1$ or $x_0*x_0*...x_0*x_2$. It just so happens that integers with usual prime factorizations of either of these forms are $\tau_7$-atoms; the proof of Theorem \ref{tau5atoms} suffices to show this point. 

However, there are other $\tau_7$-atoms that meet none of these criteria. Consider that $6$ meets none of these criteria, yet is clearly a $\tau_7$-atom, as $2\not\mu_7 3$. We wish to characterize the remaining $\tau_7$-atoms. 

\begin{theorem}\label{tau7atoms}The $\tau_7$-atoms are the usual primes, along with integers whose usual prime factorizations are of the form $7p_1p_2...p_k$ where $p_i$ is a usual prime not equal to $7$, or integers whose usual prime factorizations can be expressed as follows:

\begin{itemize*}
\item $x_0*x_0*...*x_0*x_1$\\
\item $x_0*x_0*...*x_0*x_2$\\
\item $x_1*x_2$
\end{itemize*}

where $x_0, x_1$, and $x_2$ are factors from the $\mu_7$ equivalence classes $[1], [2]$, and $[3]$, respectively.
\end{theorem}

\begin{proof}
All these cases have been addressed except integers with usual prime factorizations of the form $x_1*x_2$, yet it is evident that such an integer must not have a $\tau_7$-factorization, since both factors are already usual primes. Thus, such integers must be $\tau_7$-atoms. 

It must be shown that the criteria listed in the theorem exhaust all $\tau_7$-atoms. Thus we consider the following cases, which exhaust all possibilities:

\emph{Case 1}: Let $x$ be an integer with usual prime factorization of the form \\$x_0*x_0*...*x_0*x_1*x_1*...*x_1*x_2*x_2*...*x_2$, where there is at least one $x_0$ factor, at least one $x_1$ factor, and at least one $x_2$ factor. Notice that all $x_0$ factors can be multiplied together into a single $x_0$ factor, since $0+0=0$, and so we consider $x$ to be of the form $x_0*x_1*x_1...*x_1*x_2*x_2*...*x_2$ for the remainder of this case. 

Suppose that there are an even number of $x_1$ factors. Then since $1+1=2$, mutliply the $x_1$ factors together in pairs to make half as many $x_2$ factors. Then $x=x_0*x_2*x_2...*x_2$. Now, since $0+2=2$, then notice that $x=(x_0*x_2)*x_2*...*x_2=x_2*x_2...*x_2$, which means $x$ has a $\tau_7$-factorization; thus $x$ is not a $\tau_7$-atom.

Suppose that there are an even number of $x_2$ factors. Then since $2+2=1$ mod $3$  and $3=\frac{7-1}{2}$, then by Theorem \ref{mureps} the product $x_2*x_2$ must result in an $x_1$ factor. Thus, similar to the previous paragraph, we see that $x$ has a $\tau_7$-factorization of all $x_1$ factors, and so is not a $\tau_7$-atom.

Suppose that there are an odd number of both $x_1$ and $x_2$ factors. If there are the same number of $x_1$ and $x_2$ factors, then simply multiply each $x_1$ factor by an $x_2$ factor will yield an $x_0$ factor and so a $\tau_7$-factorization of $x$ exists. Suppose, instead, that there are more $x_1$ factors than $x_2$ factors. Then there must be at least two more $x_1$ factors than $x_2$ factors, since there are an odd number of both. Thus, following the same process of multiplying each $x_2$ factor by an $x_1$ factor until none remain, we shall see that $x=x_0*x_0*...*x_0*x_1*x_1...*x_1$. Since we can multiply all $x_0$ factors together and then multiply again by a single $x_1$ factor to yield an $x_1$ factor (as $0+0+...+0+1=1$), we see that $x$ has a $\tau_7$-factorization of all $x_1$ factors. If we assume that there are more $x_2$ factors than $x_1$ factors, this same process will provide a $\tau_7$-factorization of $x$ with only $x_2$ factors instead.

Thus, an integer with a usual prime factorization of the form in this case cannot be a $\tau_7$-atom.

\emph{Case 2}: Let $x$ be an integer with usual prime factorization of the form \\
$x_0*x_0*...*x_0*x_i*x_i*...*x_i$ for $i=1$ or $i=2$, where there are at least two $x_i$ factors. Then since $0+0+...+0+i=i$, we see that we may simply multiply all $x_0$ factors and one $x_i$ factor to yield another $x_i$ factor. Thus, $x=x_i*x_i*...*x_i$, and we see that $x$ has a $\tau_7$-factorization. Thus, an integer with a usual prime factorization of this form cannot be a $\tau_7$-atom.

\emph{Case 3}: Let $x$ be an integer with usual prime factorization of the form $x_1*x_1*...*x_1*x_2*x_2*...*x_2$, where there are at least two $x_1$ factors or at least two $x_2$ factors. If there are an even number of either $x_1$ or $x_2$ factors, then following the process of Case 1 when there are even numbers of $x_1$ or $x_2$ factors, we can simply multiply factors in pairs to produce factors of the other kind, and so can create a $\tau_7$-factorization of $x$. If there are an odd number of both $x_1$ and $x_2$ factors, then we can simply follow the process of Case 1 when there are odd numbers of $x_1$ and $x_2$ factors to create a $\tau_7$-factorization of $x$. Thus, integers with usual prime factorizations of this form must not be $\tau_7$-atoms.
\end{proof}

These cases exhaust all possibilities; thus, the cases outlined in the statement of the theorem must be the only ones which are $\tau_7$-atoms. Our goal of finding patterns or generalizations of $\tau_n$-atoms is being developed, and is looking more likely, as we see that forms of usual prime factorizations that are the same as those of the $\tau_n$-atoms for lower usual prime values of $n$ still yield $\tau_n$-atoms for higher usual prime values of $n$. We shall examine one more usual prime value of $n$ before we present our findings.

\end{section}
\begin{section}{The $\tau_{11}$-atoms}
\quad The next usual prime is $11$, and this is the largest jump we have made thus far. Previously, we have only increased the value of $n$ by either one or two; this time we advance by four. As a result, we also gain more $\mu_{11}$ equivalence classes: other than $[0]$, we now have $[1], [2], [3], [4],$ and $[5]$. This makes things far more complicated. We shall denote elements of $[1]$ by $x_0$, as usual, and we denote elements of $[2^i]$ by $x_i$ where $1\leq i\leq 4$. 

As always, we know that the usual primes are $\tau_{11}$-atoms, and we know by Theorem \ref{onefactoratom} that integers with usual prime factorizations of the form $11p_1p_2...p_k$ where $p_i$ is a usual prime not equal to $11$ are $\tau_{11}$-atoms. Given our observations thus far, it is also reasonable to ask whether integers with prime factorizations of the form $x_0*x_0*...*x_0*x_i$ where $i\neq 0$ and of the form $x_i*x_j$ where $i\neq j$ are $\tau_{11}$-atoms, since they are $\tau_n$-atoms for usual prime $n<11$, and indeed we see they are, for the same reasons outlined in the previous sections. In fact, in the latter case we can even include $x_0$ factors as well, so long as the product $x_i*x_j$ does not yield an $x_0$ factor; in other words, if we have usual primes $x_i$ and $x_j$ such that $i+j\neq 0$ (mod $5$), then $x_0*x_0*...*x_0*x_i*x_j$ is a $\tau_{11}$-atom. However, yet again we see that there must be other conditions which can be met to produce a $\tau_{11}$-atom; it just so happens that $50$ is a $\tau_{11}$-atom, but does not meet any of the aforementioned criteria. Nor does $296$, another $\tau_{11}$-atom. In fact, the usual prime factorizations of these two integers do not even have much in common, with $50=2*5*5=x_1*x_4*x_4$ and $296=2*2*2*37=x_1*x_1*x_1*x_2$. 

\begin{theorem}\label{tau11atoms}
The $\tau_{11}$-atoms are the usual primes, along with integers whose usual prime factorizations are of the form $11p_1p_2...p_k$ where $p_i$ is a usual prime not equal to $11$, or integers whose usual prime factorizations can be expressed as follows:
\begin{itemize*}
\item $x_0*x_0*...*x_0*x_i$ where $i\neq 0$\\
\item $x_i*x_j$ where $i\neq j$\\
\item $x_0*x_0*...*x_0*x_i*x_j$, where $i\neq j$ and $i+j\neq 0$ (mod $5$)\\
\item $x_i*x_i*x_j$ where $0\neq i, j$ (mod $5$), $i\neq j$ and $2i\neq j$ (mod $5$)\\
\item $x_0*x_0*...*x_0*x_i*x_i*x_j$, where $i\neq j$, $2i\neq j$ (mod $5$), and $2i+j\neq 0$ (mod $5$)\\
\item $x_i*x_i*x_i*x_{2i(\emph{mod } 5)}$ where $i\neq 0$ (mod $5$).
\end{itemize*}
\end{theorem}

\begin{proof}As a first step toward discovering all types of $\tau_{11}$-atoms, a table was formed which detailed all possible combinations of $\mu_{11}$ equivalence class representatives (other than $x_0$) with at most four representatives per class. This table was manually checked, case by case, for which combinations were reducible and which were not. The table itself is not listed here due to its size: there are a total of 625 cases to check. 

Of the irreducible elements, most can be readily described by the cases of the $\tau_7$-atoms, namely the first 5 cases listed in the statement of the theorem above (from usual primes through $x_0*x_0*...x_0*x_i*x_j$). At this point, then, we briefly show that the last 3 cases do indeed describe $\tau_{11}$-atoms.

\underline{$x_i*x_i*x_j$ where $0\neq i, j$ (mod $5$), $i\neq j$ and $2i\neq j$ (mod $5$)}: Since $2i\neq j$ (mod $5$), then clearly $x_{2i}*x_j$ cannot be a $\tau_{11}$-factorization; thus, consider the only alternative: $x_{i+j}*x_i$. Suppose, by way of contradiction, that this is a $\tau_{11}$-factorization. Then $i+j=i$ (mod $5$), and so $j=0$ (mod $5$). But $j\neq 0$ (mod $5$); here arises a contradiction. Thus, an integer with a usual prime factorization of this type must be a $\tau_{11}$-atom. The proof of the next case is clearly covered by this one with the addition of the final condition, $2i+j\neq 0$ (mod $5$).

\underline{$x_i*x_i*x_i*x_{2i(\emph{mod } 5)}$ where $i\neq 0$ (mod $5$)}: Since $ni\neq mi$ (mod $5$) for all $n\neq m$ where $0<n,m<5$, any arrangement of these factors clearly does not result in a $\tau_{11}$-factorization. Thus, an integer with a usual prime factorization of this type must be a $\tau_{11}$-atom.

Now we must show that this list of cases is exhaustive. As the initial table only considered sets of factors where there were at most 4 representatives per $\mu_{11}$ equivalence class (other than $x_0$), it is necessary to prove that any combination of factors involving more than 4 representatives per $\mu_{11}$ equivalence class must not be an atom. First, note that any product of 5 elements from the same $\mu_{11}$ equivalence class must yield an $x_0$ element, since $5i \equiv 0$ (mod $5)$. Take some arbitrary combination of $\mu_{11}$ representatives already considered (that is, one with at most $4$ representatives per equivalence class other than $x_0$), and call its product $x$. Then, for each $i = 1, 2, 3, 4$, consider $y_i=x*x_i*x_i*x_i*x_i*x_i$. It is evident that, if $x$ was not a $\tau_{11}$-atom, then $y_i$ must not be for any value of $i$, since the product of the $x_i$ terms produces an $x_0$ term, which can be trivially absorbed into the $\tau_{11}$-factorization of $x$. Thus, we are only concerned with the case when $x$ is a $\tau_{11}$-atom; that is, when the factorization of $x$ meets any of the criteria listed in the theorem. In this case it is easy to show (again, by exhaustion) that, regardless of $i, y_i$ is not a $\tau_{11}$-atom; again, an explicit proof is not listed here in consideration of length. This means, then, that any factorization involving $5$ or more representatives from a single $\mu_{11}$ equivalence class other than $x_0$ must nessarily be reducible. Thus, the cases listed in the theorem must exhaust all $\tau_{11}$ atoms. 
\end{proof} 

It seems that as $n$ increases, the number of cases of usual prime factorizations that result in $\tau_n$-atoms generally increases as well. This is bothersome, and so we wish to generalize our findings thus far in the hope of finding a way to generalize all $\tau_n$-atoms in the future.
\end{section}
\begin{section}{Generalization to $\tau_n$-atoms}
To cut right to the chase, we were indeed able to generalize our findings in a satisfactory way:

\begin{theorem}\label{generalization}
Let $n$ be an odd prime, and let $y$ be an integer with a usual prime factorization of the form
$$
y=x_{01}*x_{02}*...*x_{0k}*x_{i1}*x_{i2}*...*x_{im}*x_j,
$$
where $x_{ab}\in[x_a], x_{ab}\neq n$ for all $a, b$, and $0\neq i, j$. 
\\ \\
Then $y$ is a $\tau_n$-atom if the following conditions hold:
\begin{enumerate*}
\item if, for some integer $z$, $z|m$, then $zi\not\equiv j$ (mod $\frac{n-1}{2}$),\\
\item if $mi+j\equiv 0$ (mod $\frac{n-1}{2}$), then $k=0$, and\\
\item $ci+j\not\equiv 0$ (mod $\frac{n-1}{2}$) for any $c<m-1$.
\end{enumerate*}
\end{theorem}
\begin{proof}
There are only finitely many ways in which $y$ might fail to be an atom. If one were trying to properly $\tau_n$-factor $y$, one might first try to create a proper $\tau_n$-factorization involving terms already present in the usual prime factorization of $y$, namely $x_0, x_i,$ or $x_j$ factors.

Starting with the possibility of a proper $\tau_n$-factorization involving all $x_0$ terms, clearly this would imply that $mi+j\equiv 0$ (mod $\frac{n-1}{2}$), and so by condition $2$ we see that $k=0$; that is, there are no standard prime factors of $y$ in the $\mu_n$ equivalence class $[x_0]$. Since $j\neq 0$, then, it must be necessary to take the product of $x_j$ with some number of terms from $[x_i]$ in order to produce a single term in $[x_0]$, but condition $3$ maintains that it will require at least $m-1$ such terms. Thus, we are either left with a factorization of the form $x_0*x_i$ or merely $x_0$, and since $i\neq 0$ then both cases fail to produce a proper $\tau_n$-factorization. Since a $\tau_n$-factorization of terms in $[x_0]$ is impossible, then, we shall ignore all terms in $[x_0]$ henceforth, as they are absorbed into any arbitrary product. 

Next, let us consider the possibility of a proper $\tau_n$-factorization of terms all in $[x_i]$. Again, by condition $3$, this case can quickly be discounted. 

Finally, consider the possibility of a proper $\tau_n$-factorization of terms all in $x_j$. Then clearly we must multiply sets of terms from $[x_i]$ in order to obtain at least one term from $[x_j]$. Suppose the minimum number of terms is $d$; that is, if $ki\equiv j$ (mod $\frac{n-1}{2}$) then $k\geq d$. Due to condition $1$, then, we know that $d\nmid m$; thus, there must be some leftover number of terms from $[x_i]$ which must produce an $x_j$ term. Call this positive integer $c\neq d$. \\
Suppose $c>d$. Then $c=d+r$ for some integer $r>0$. Then we see that $j\equiv di\equiv ci\equiv (d+r)i$ (mod $\frac{n-1}{2})$, and so $di\equiv (d+r)i$ (mod $\frac{n-1}{2})$. Thus, $ri\equiv 0$ (mod $\frac{n-1}{2})$. Suppose, then, that $r<d$. Then $(d-r)i\equiv di\equiv j$ (mod $\frac{n-1}{2})$; yet $d-r<d$, and $d$ is the minimum number of $x_i$ terms such that their product is equivalent to $j$, and so a contradiction arises. Thus, $r$ must be greater than $d$. Let $r-d=s$. Then $ri\equiv (d+s)i\equiv 0$ (mod $\frac{n-1}{2})$, and since $di\equiv j$ (mod $\frac{n-1}{2})$, then this means $si+j\equiv 0$ (mod $\frac{n-1}{2})$. But certainly, since $s<d$ and $d<m$ then $s<m-1$, and this is in violation of condition $3$. Thus, we must conclude that $c<d$. Then, since $d$ is the minimum number of $x_i$ terms necessary to produce an $x_j$ term, $ci$ must be equivalent to $0$ (mod $\frac{n-1}{2})$. Notice, though, that this implies that $di\equiv di-ci\equiv (d-c)i\equiv j$ (mod $\frac{n-1}{2})$, and clearly $d-c<d$; yet, again, $d$ is the minimum number of $x_i$ terms required to produce an $x_j$ term. Again, a contradiction arises, and thus a $\tau_n$-factorization of $y$ involving terms only from $[x_j]$ must be impossible.


Therefore, if $y$ is to have a proper $\tau_n$-factorization, it must be composed of terms all from some $\mu_n$ equivalence class other than $[x_0], [x_i],$ and $[x_j]$. Call this class $[x_g]$. Then suppose $y=x_{g1}*x_{g2}*...*x_{gz}$ for some positive integer $z$. Then $mi+j\equiv zg$ (mod $\frac{n-1}{2}$). Further, since condition $3$ does not permit us to multiply $x_j$ with some number of $x_i$ terms to yield an $x_0$ term and still create such a factorization, we see that some product involving $x_i$ terms and $x_j$ terms must produce an $x_g$ term; that is, $(m-a)i+j\equiv g$ (mod $\frac{n-1}{2}$) for some positive integer $a<m$. The remainder of the $x_i$ terms must account for all the remaining $x_g$ terms; that is, $ai\equiv (z-1)g$ (mod $\frac{n-1}{2}$). Now, as we are attempting to separate these $a$ terms into $(z-1)$ groups, each of which are multiplied to yield a single $x_g$ term, we have two options: either these $(z-1)$ groups all have the same number of $x_i$ terms, or they do not. Suppose they do not; then select two groupings which have differing numbers of $x_i$ terms, say $b>c$. Then $bi\equiv ci\equiv g ($mod $\frac{n-1}{2})$. But then $(b-c)i\equiv 0($mod $\frac{n-1}{2})$, and so we see that we can simply absorb $(b-c)$ terms from the larger grouping into a single $x_0$ term, and absorb that into our previous product of $(m-a)$ $x_i$ terms and the single $x_j$ term; in short, we will simply increase the value of $a$ until we have $(z-1)$ groupings of an equal number $x_i$ terms, each of which multiply to produce a single $x_g$ term. Returning to our previous equivalence $ai\equiv (z-1)g($mod $\frac{n-1}{2})$, since the groupings of $x_i$ factors have an equal number of terms, this implies that $(z-1)|a$. Thus $\frac{a}{z-1}$ must be an integer. Then $\frac{a}{z-1}i\equiv g$ (mod $\frac{n-1}{2}$); thus by transitivity $\frac{a}{z-1}i\equiv (m-a)i+j$ (mod $\frac{n-1}{2}$). Then $(m-\frac{az}{z-1})i+j\equiv 0$ (mod $\frac{n-1}{2})$, so by condition $3$ we see that $\frac{az}{z-1}$ must be either $1$ or $0$.

\emph{Case 1}: $\frac{az}{z-1}=0$. Then either $a=0$ or $z=0$, but both $a$ and $z$ are strictly positive; a contradiction arises.

\emph{Case 2}: $\frac{az}{z-1}=1$. Then $az=z-1$, so $az<z$. This implies that $a<1$. But $a$ is a positive integer, so $a\geq 1$. Again, we see a contradiction. 

Therefore, no $\tau_n$-factorization of $y$ can exist, and so $y$ must be a $\tau_n$-atom.
\end{proof}

Looking back at our previous theorems, one can see that this generalization does indeed cover all $\tau_n$-atoms mentioned in the paper other than the usual primes and integers of the form $np_1p_2...p_k$. We are pleased with this result, but more work is required before a complete generalization can be found. We stopped at $n=11$ for a reason: $n=13$ provides a truly difficult challenge, with atoms that can involve factors from at least $4$ different $\mu_{13}$ equivalence classes. Moreover, our approach to $\tau_{11}$ involving the spreadsheet doesn't seem applicable to $\tau_{13}$: while the $\tau_{11}$ sheet involved the more manageable number of $625$ cases, $\tau_{13}$ would involve $7776$ cases. We are very interested in the possibility of writing a program to generate and factor these cases far more efficiently, but until such a program is available to us we are left with the options of finding a new way to approach $\tau_{13}$ or simply enduring all those cases.
\end{section}

\end{document}